\documentclass[12pt,reqno]{amsart}
\pdfoutput=1
\usepackage{graphicx,amsmath,amsthm,verbatim,amssymb,lineno,titletoc}
\usepackage{bbm}
\usepackage{hyperref}
\hypersetup{colorlinks}
\newcommand{\arxiv}[1]{{\tt \href{http://arxiv.org/abs/#1}{arXiv:#1}}}

\newcommand{\old}[1]{}
\newcommand{\moniker}[1]{{\em (#1)}}
\newcommand{\dmoniker}[1]{(#1)}

\DeclareRobustCommand{\SkipTocEntry}[5]{}

\newtheorem{theorem}{Theorem}
\newtheorem{prop}[theorem]{Proposition}

\newtheorem{lemma}[theorem]{Lemma}
\newtheorem{corollary}[theorem]{Corollary}
\newtheorem{conjecture}[theorem]{Conjecture}

\theoremstyle{remark}

\numberwithin{counter}{section}

\theoremstyle{definition}
\newtheorem{definition}{Definition}

\def\av{\mathrm{av}}

\def\Stab{\mathcal{S}}
\def\Wave{\mathcal{W}}

\def\Rec{\mathrm{Rec}}

\def\one{\mathbf{1}}

\def\N{\mathbb{N}}
\def\Z{\mathbb{Z}}

\def\EE{\mathbb{E}}
\def\PP{\mathbb{P}}
\def\eps{\epsilon}

\begin{document}

\title[Threshold state]{Threshold state and a conjecture of Poghosyan, Poghosyan, Priezzhev and Ruelle}

\author[Lionel Levine]{Lionel Levine}

\address{Department of Mathematics, Cornell University, Ithaca, NY 14853. {\tt \url{http://www.math.cornell.edu/~levine}}}
\thanks{The author was partly supported by NSF grant DMS-1243606.}

\date{May 20, 2014}

\begin{abstract}
We prove a precise relationship between the threshold state of the fixed-energy sandpile and the stationary state of Dhar's abelian sandpile: In the limit as the initial condition $s_0$ tends to $-\infty$, the former is obtained by size-biasing the latter according to \emph{burst size}, an avalanche statistic. The question of whether and how these two states are related has been a subject of some controversy since 2000.

The size-biasing in our result arises as an instance of a Markov renewal theorem, and implies that the threshold and stationary distributions are not equal even in the $s_0 \to -\infty$ limit. We prove that nevertheless in this limit the total amount of sand in the threshold state converges in distribution to the total amount of sand in the stationary state, confirming a conjecture of Poghosyan, Poghosyan, Priezzhev and Ruelle.
\end{abstract}

\maketitle

\section{Introduction}


How much memory does a critical system retain of its pre-critical past? This is the question lurking beneath the prediction \cite{VDMZ2} that $\zeta_{s} = \zeta_{\tau}$: the stationary density of the abelian sandpile should equal the threshold density of the fixed energy sandpile (the density of sand at which it becomes permanently unstable; precise definitions are given below in \textsection\ref{s.laplacian}--\ref{s.densities}). In \cite{FLW, FLW2} the above prediction was refuted on a few simple graphs where $\zeta_{\tau}$ can be computed exactly, and simulations on the two-dimensional torus $\Z_N \times \Z_N$ 
show that $\zeta_{\tau} \approx 2.125288$ differs slightly from 
$\zeta_s = 2.125000$ (the exact evaluation $\zeta_s = 17/8$ was recently proved in \cite{PPR,KW}). 

Why are these values so close if they are not equal? Jo and Jeong \cite{JJ} pointed out that $\zeta_{\tau}$ depends on the initial condition: the threshold state cannot be expected to have universal properties because it retains some memory of its pre-critical past.
Poghosyan et al.\ \cite{PPPR} performed numerical experiments to estimate $\zeta_{\tau}(h)$ for constant initial conditions $h \equiv 2,1,0,-1,-2$. On the basis of those experiments they conjectured the following. 
\begin{conjecture} \moniker{Poghosyan, Poghosyan, Priezzhev, Ruelle \cite{PPPR}}
\label{c.pppr}
	$\zeta_{\tau}(h) \to \zeta_s$ as $h \to -\infty$.
\end{conjecture}

This conjecture is natural if one suspects the inequality $\zeta_{\tau}(h) \neq \zeta_s$ for finite $h$ arises from insufficient mixing. Starting from a highly subcritical state $h \ll 0$ allows for enough time to mix before reaching the critical threshold, so that the threshold state
can be compared to the stationary state.  We will see that this basic intuition is correct except for a size-biasing adjustment arising from the fact that the threshold is reached at a random time.
 
The abelian sandpile is often proposed as a model for other self-organized critical systems.  The significance of Conjecture~\ref{c.pppr} is that it suggests specifically which slowly driven systems are good candidates for comparison with the stationary state of the abelian sandpile: those which have sufficient time to mix during the driving phase so that they forget their initial subcritical state.

\subsection{Outline}
In this paper we relate not just the densities $\zeta_{\tau}$ and $\zeta_s$ but the actual distributions of the threshold and stationary states.
Our main result is Theorem~\ref{t.main} below. It gives the limiting joint distribution as $h \to -\infty$ of the epicenter (the last vertex at which sand is added, triggering a system-spanning avalanche) and the recurrent representative of the threshold state. The law of the latter turns out to be a size-biasing of the stationary law by \emph{burst size}, an avalanche statistic we define below. The mechanism for size-biasing is a Markov renewal theorem, Proposition~\ref{t.marathon}. We use these results to prove Conjecture~\ref{c.pppr}.


\subsection{Graph Laplacian; Stabilizability; Odometer function}
\label{s.laplacian}

Let $G=(V,E)$ be a finite directed graph, with multiple edges permitted. We assume throughout that $G$ is connected and \emph{Eulerian}, that is, each vertex $i \in V$ has the same number $\deg(i)$ of incoming edges as outgoing edges.  In particular, any undirected graph can be made Eulerian by replacing each edge with a pair of oppositely oriented directed edges.  The \emph{graph Laplacian} $\Delta$ acts on functions $u : V \to \Z$ by
	\begin{equation} \label{e.laplacian} \Delta u(i) = - \deg(i) u(i) + \sum_e u(e^-) \end{equation}
where the sum is over incoming edges $e$ to vertex $i$, and $e^-$ denotes the other endpoint of the edge.

A \emph{sandpile} is a function $s: V \to \Z$. We think of a positive value $s(i)>0$ as a number of sand grains (or ``chips'') at $i$, and negative value as a ``hole'' that can be filled by chips.
Vertex $i$ is \emph{unstable} if $s(i) \geq \deg(i)$, and an unstable vertex \emph{topples} by sending away $\deg(i)$ chips, one along each outgoing edge. Note that toppling $i$ yields the sandpile $s + \Delta \delta_i$, where $\delta_i(j) = \one \{i=j\}$. We say that $s$ is \emph{stabilizable} if starting from $s$ there exists a finite sequence of topplings of unstable vertices resulting in a sandpile $\Stab(s) \leq \deg -1$ (inequalities between functions are coordinatewise).  This $\Stab(s)$ is called the \emph{stabilization} of $s$, and satisfies
	\[ \Stab(s) = s + \Delta u \]
where $u(i)$ is the number of times $i$ topples. This function $u$ is called the \emph{odometer} of $s$. Both $u$ and $\Stab(s)$ depend only on $s$, and not on the choice of toppling sequence. (For the proof of this \emph{abelian property} and other basic facts about sandpiles stated here without proof, see one of the surveys \cite{Dhar06,HLMPPW,Jarai}.)

Denote by \[ |s| = \sum_{i \in V} s(i) \] the total number of chips in $s$, a quantity conserved under toppling. If $|s| > 
|\deg - 1|$ then there will always be a vertex $i$ with at least $\deg(i)$ chips, so $s$ is not stabilizable. However, stabilizability depends not just on the total number of chips but on how they are arranged. For instance, on the complete graph on $n$ vertices $\{0,1,\ldots,n-1\}$ the sandpile $s(i)=i$ is not stabilizable, but the stable sandpile $s \equiv n-2$ has more chips if $n \geq 4$.

\subsection{The closed chain}
\label{s.closed}
For each $i \in V$ we define an \emph{addition operator} $a_i$ which acts on a sandpile $s$ by adding one chip at $i$ and then stabilizing if possible,
	\[ a_i s = \begin{cases} \Stab (s+\delta_i) & \text{if $s+\delta_i$ is stabilizable} \\
					s+\delta_i & \text{else}.
					\end{cases} \]
The \emph{closed chain} (sometimes called the fixed-energy sandpile) is a Markov chain $(s_k)_{k\geq 0}$ on the space $V^\Z$ of sandpiles.
Given an initial state $s_0$ and a probability distribution $\alpha$ on $V$, the states $s_k$ for $k\geq 1$ are defined by 
	\[ s_{k} = a_{i_k} s_{k-1} \]
where $i_1, i_2, \ldots$ are independent random draws from $\alpha$. Thus, at each discrete time step we add a sand grain at a random site and then stabilize if possible.  Usually $\alpha$ is taken to be the uniform distribution, $\alpha_i \equiv 1/\#V$. We will not need to assume this, but we do assume throughout that there is a positive probability of dropping sand at each vertex: $\alpha_i>0$ for all $i \in V$. 

The \emph{threshold} is defined as the random time
	\begin{equation} \label{e.threshold} \tau = \tau(s_0) = \min \{k \geq 0 \,:\, s_k \text{ is not stabilizable}\}. \end{equation}
  We are interested in the distribution of
 \begin{itemize}
\item The \emph{threshold state} $s_{\tau}$.
\item The \emph{epicenter} $i_\tau$, or ``straw that breaks the camel's back.''
  \end{itemize}

\subsection{The open chain}
\label{s.open}
In order to analyze $s_\tau$ and $i_\tau$ we define a second Markov chain $(\rho_k)_{k \geq 0}$ in which stabilization takes place with respect to a fixed sink vertex $z \in V$. Its addition operators $\hat{a}_i$ are defined by
	\[ \hat{a}_i \rho = \Stab_z (\rho+\delta_i). \]
The subscript $z$ means that chips entering $z$ disappear from the system, and that $z$ is forbidden to topple. Because $G$ is connected, every sandpile is stabilizable with respect to $\Stab_z$. Thus, the definition of $\hat{a}_i$ unlike that of $a_i$ does not require a second case. The word \emph{stabilizable} in this paper will always mean ``stabilizable with respect to $\Stab$.''

The \emph{open chain} (often called the BTW sandpile \cite{BTW} or Dhar's abelian sandpile \cite{Dhar90}) with initial state $\rho_0$ is defined for $k\geq 1$ by
	\[ \rho_k = \hat{a}_{i_k} \rho_{k-1} \]
where $i_1, i_2, \ldots$ are independent random draws from $\alpha$ as in the closed chain.  Note that if $i_k=z$ then $\rho_k = \rho_{k-1}$. 
Dhar's burning test \cite{Dhar90} identifies the recurrent states of the open chain on an Eulerian graph. 

\begin{definition}
\label{d.zrec}
\dmoniker{Burning Test}
For $z \in V$, a sandpile $\rho$ is \emph{$z$-recurrent} if $\rho(z) = \deg(z)$ and $\rho(i) \leq \deg(i)-1$ for all $i\in V-\{z\}$ and every site in $V-\{z\}$ topples exactly once during the stabilization of $\rho + \Delta \delta_z$ with respect to $\Stab_z$.
\end{definition}

The convention $\rho(z)=\deg(z)$ plays no role in the dynamics of the open chain, but it will be convenient when comparing to the closed chain.  We denote the set of $z$-recurrent sandpiles by $\Rec(z)$ and its cardinality by $\kappa$.

The $z$-recurrent sandpiles are in bijection with spanning trees of $G$ oriented toward~$z$ \cite{HLMPPW,MD}. For an Eulerian graph $G$, the BEST theorem 
relating spanning trees to Eulerian tours implies that the number of such trees does not depend on $z$ \cite[Corollary 5.6.3]{Stanley}. 
Thus, $\kappa$ does not depend on $z$.

Each operator $\hat{a}_i$ acts as a permutation on $\Rec(z)$ \cite{Dhar90}. It follows that the stationary distribution $\pi$ of the open chain is uniform: $\pi(\rho)=1/\kappa$ for all $\rho \in \Rec(z)$. These observations about sandpiles are an instance of a general mechanism by which group actions arise from finite commutative monoid actions \cite[Lemma A.4]{BL}.

\subsection{The open chain as a factor of the closed chain}
In Lemma~\ref{l.rectest} we will show that any sandpile $s$ has a unique \emph{$z$-recurrent decomposition}
	\begin{equation} \label{e.zrec.intro} s = \rho + m\delta_z + \Delta v \end{equation}
where $\rho \in \Rec(z)$ and $m \in \Z$ and $v(z) = 0$. Moreover, $s$ is stabilizable if and only if $m<0$. 
This last assertion follows very easily from well-known facts about sandpiles, but it is key to our approach. Denoting by $R_z(s)$ the unique $\rho \in \Rec(z)$ satisfying \eqref{e.zrec.intro}, it is straightforward to check (Lemma~\ref{l.intertwining}) that
	\begin{equation} \label{e.intertwining} R_z (a_i s) = \hat{a}_i R_z (s). \end{equation}
Consider now the closed chain $s_k$ and its decomposition
	\begin{equation} \label{e.zrec.closed} s_k = \rho_k + m_k\delta_z + \Delta v_k. \end{equation}
From \eqref{e.intertwining} and $s_k = a_{i_k} s_{k-1}$ we see by induction that $\rho_k = \hat{a}_{i_k} \rho_{k-1}$, so $\rho_k = R_z(s_k)$ follows the law of the open chain.  The open chain is thus a deterministic function of the closed chain.

\subsection{Main theorem}

In \textsection\ref{s.proof} we prove the following.

\begin{theorem}
\label{t.main}
Let $(s_k)_{k \geq 0}$ be the closed chain and \eqref{e.zrec.closed} its $z$-recurrent decomposition. Let $\tau = \tau(s_0)$ be the threshold time \eqref{e.threshold}. 
As $|s_0| \to -\infty$, the joint distribution of $(i_\tau,\rho_\tau,m_\tau)$ converges to
	\begin{equation} \label{e.main} \PP_{s_0} \{ (i_\tau,\rho_\tau,m_\tau) = (i,\rho,m) \} \to \frac{\alpha_i}{\kappa} \one \{0 \leq m \leq |\hat{a}_i^{-1} \rho| - |\rho| \}. \end{equation}
\end{theorem}

The meaning of the limit $|s_0| \to -\infty$ is the following: For any $\eps>0$ there exists $K<0$ such that for any initial configuration $s_0$ satisfying $|s_0| < K$ and any $i \in V,\rho \in \Rec(z), m \in \N$, the left and right sides of \eqref{e.main} differ by at most $\eps$. Theorem~\ref{t.main} expresses a kind of universality, in the sense that the limiting distribution does not depend on the nature of the initial state $s_0$ as long as its total chip count $|s_0|$ tends to~$-\infty$.

The remainder of this section explores some corollaries of Theorem~\ref{t.main}. 

\subsection{Distribution of the epicenter}

Taking $z=i$ in Theorem~\ref{t.main} and noting that in this case $\hat{a}_i^{-1} \rho = \rho$, we obtain the following.

\begin{corollary}
\label{c.sinkatsource}
For each $i \in V$ and $\rho \in \Rec(i)$ we have as $|s_0| \to -\infty$
	\[ \PP_{s_0} \{ i_\tau = i,\, R_{i}(s_\tau) = \rho \} \to \frac{\alpha_i}{\kappa}. \]
\end{corollary}
Summing over $\rho$
we obtain 

\begin{corollary}
\label{c.straw}
For each $i \in V$ we have as $|s_0| \to -\infty$
	\[ \PP_{s_0} \{ i_\tau = i \} \to \alpha_i. \]
\end{corollary}

In other words, the epicenter $i_\tau$ has the same distribution as the input. We have derived this result for a very particular model, the fixed-energy sandpile, but we would like to suggest it as a general principle: \emph{In a system driven slowly to criticality from a highly subcritical initial state, stress is distributed uniformly} in the sense that
the probability of triggering a system-spanning avalanche by applying additional stress does not depend on where the additional stress is applied.

\subsection{Comparison of densities}
\label{s.densities}

We now give precise definitions of the densities $\zeta_s$ and $\zeta_\tau$ appearing in Conjecture~\ref{c.pppr}.

\begin{definition}
The \emph{stationary density} $\zeta_s$ is the expected number of chips (per site) in a stationary state of the open chain,
	\[ \zeta_s := \frac{1}{\kappa} \sum_{\rho \in \Rec(z)} \frac{|\rho|}{\# V}. \]
\end{definition}

Recall that $|\rho| = \sum_{i \in V} \rho(i)$, that this sum includes the sink $i=z$, and that $\rho(z)=\deg(z)$ by definition. In the case of an undirected graph $G$, Merino's theorem \cite{Merino} implies that $\zeta_s = e+ \frac{\partial}{\partial y} T(x,y) |_{x=y=1}$ where $e$ is the number of (undirected) edges in $G$ and $T$ is the Tutte polynomial of $G$. 

Recently, Perrot and Pham  \cite{PP} have generalized Merino's theorem to Eulerian graphs. They show that if $G$ is Eulerian then for each $n \in \N$, the number of $\rho \in \Rec(z)$ such that $|\rho|=n$ does not depend on $z$. In particular, $\zeta_s$ for an Eulerian graph does not depend on the choice of sink $z$.

\begin{definition}
The \emph{threshold density} $\zeta_{\tau}(s_0)$ of an initial state $s_0$ is the expected number of chips (per site) in the corresponding threshold state $s_\tau$,
	\[ \zeta_{\tau}(s_0) := \EE_{s_0} \frac{|s_\tau|}{\# V}. \]
Here the expectation is taken over the random additions defining the closed chain $(s_k)_{k \geq 0}$.
\end{definition}

In \cite{PPPR} Conjecture~\ref{c.pppr} was posited for the $N \times N$ torus graph, but we will show that it holds on any finite Eulerian graph.
We will also strengthen it in two ways: First, instead of requiring the initial configuration to be a constant $s_0 \equiv h$ tending to $-\infty$, we require only that the total chip count $|s_0|$ tends to $-\infty$. Second, instead of taking expectations we compare the actual random number of chips in the threshold and stationary states. 

\begin{corollary}
\label{c.densities}
For any $n \in \N$ and $z \in V$ we have as $|s_0| \to -\infty$,
	\[ \PP_{s_0}(|s_\tau| = n) \to \frac{1}{\kappa} \sum_{\rho \in \Rec(z)} \one \{ |\rho|=n \}. \]
\old{ 
Fix $n \in \N$. On any Eulerian graph $G$, as $|s_0| \to -\infty$,
	\[ \PP_{s_0}(|s_\tau| = n) \to \sum_{i \in V} \frac{\alpha_i}{\kappa} \sum_{\rho \in \Rec(i)} \one \{ |\rho|=n \}. \]
On any undirected graph $G$, for any $i \in V$ we have as $|s_0| \to -\infty$,
	\[ \PP_{s_0}(|s_\tau| = n) \to \frac{1}{\kappa} \sum_{\rho \in \Rec(i)} \one \{ |\rho|=n \}. \]
}
\end{corollary}

This corollary is proved in \textsection\ref{s.proof}. 
We remark that Theorem~\ref{t.main} and its corollaries are exact results on a finite Eulerian graph of fixed size: The only limit taken is the initial condition $|s_0| \to -\infty$.  

Corollary~\ref{c.densities} (convergence in distribution) implies Conjecture~\ref{c.pppr} (convergence in expectation) because the graph $G$ is fixed and $0 \leq |s_\tau| \leq \# E $.

\subsection{Burst size}
The bound on $m$ on the right side of \eqref{e.main} is best understood as measuring the size of the avalanche caused by adding a chip at $i$ to the recurrent state $\hat{a}_i^{-1} \rho$. Since this quantity will appear often, we make the following definition. 

\begin{definition} 
\label{d.burst}
For $\rho \in \Rec(z)$ and $i\in V$ we define the \emph{burst size}
	\[ \av_{i \to z}(\rho) := |\hat{a}_i^{-1} \rho| - |\rho| + 1. \]
\end{definition}

Equivalently, $\av_{i \to z}(\rho)$ is the number of chips that fall into the sink $z$ during the stabilization of $\hat{a}_i^{-1} \rho + \delta_i$ to $\rho$. The dependence on $z$ is via the operator $\hat{a}_i$. 
Note that $\av_{z \to z}(\rho) = 1$ for all $\rho \in \Rec(z)$ (if a chip is dropped directly into the sink, then no toppling occurs and the burst size is $1$).

Summing the right side of \eqref{e.main} over $i$ and $m$ we obtain the following.
\begin{corollary} For any $z \in V$ and $\rho \in \Rec(z)$, we have as $|s_0| \to -\infty$
\label{c.marginal}
	\begin{equation} \label{e.sizebiased} \PP_{s_0} \{ R_z(s_\tau) = \rho \} \to \frac{1}{\kappa} \sum_{i \in V} \alpha_i \av_{i \to z}(\rho). \end{equation}
\end{corollary}
Thus the distribution of the $z$-recurrent representative of the threshold state is the size-biasing of the uniform distribution $1/\kappa$ by average burst size. It is interesting to compare this result with Corollary~\ref{c.sinkatsource}: if instead of a fixed sink $z$ we place the sink at the (random) epicenter $i_\tau$, then the size-biasing disappears. 

Denoting by $\theta_z(\rho)$ the right side of \eqref{e.sizebiased}, we remark that $\theta_z$ can differ considerably from the uniform distribution on $\Rec(z)$. For example, taking $G$ to be the complete graph on vertices $\{0,1,\ldots,n-1\}$ with $z=n-1$ and $\alpha_i \equiv 1/n$, the maximal recurrent configuration $\rho_{\max} = n-2 +\delta_z$ has $\av_{i\to z}(\rho)=0$ unless $i=z$, so $\theta_z(\rho_{\max}) = \frac{1}{n\kappa}$. By contrast the minimal recurrent configuration $\rho_{\min}(i)=i$ has $\theta_z(\rho_{\min}) = ( \frac{n(n-1)}{2}+1 ) \frac{1}{n\kappa}$.

\subsection{Idea of the proof}
Why does burst size appear in Theorem~\ref{t.main}? The increments of the process $m_k$ of \eqref{e.zrec.closed} are burst sizes: $m_{k} - m_{k-1} = \av_{i_{k} \to z}(\rho_k)$. 
Moreover (as will be proved in Lemma~\ref{l.rectest}) $s_k$ is stabilizable if and only if $m_k < 0$, so the threshold time can be expressed as
	\begin{equation} \label{e.thresholdm} \tau = \min \{ k \,:\, m_k \geq 0 \}. \end{equation}
If $m_0$ is very negative, so that a long time must pass until $m_k \geq 0$, then the process $m_k$ is more likely to cross $0$ during a large jump than a small jump. 
This idea is formalized in the Markov renewal theorem (Proposition~\ref{t.marathon}).

The name ``burst size'' is inspired by Dhar's survey \cite{Dhar06} in which he characterizes self-organized critical systems as those in which ``\emph{the build-up of stress...\ is a slow steady process, but the release of stress occurs sporadically in bursts of various sizes}.'' Earthquakes, forest fires, avalanches, rainfall and financial market crashes are some examples of bursts in such systems.

We are not aware of any systematic study of burst size in the abelian sandpile. More commonly studied measures of avalanche size include the total number of topplings, the volume or diameter of the set of sites that topple, and the time to relax if topplings are carried out in parallel. Avalanches can be decomposed into smaller toppling events called ``waves,'' and there is a kind of duality between waves of positive burst size and waves occurring last in an avalanche; see Table~\ref{table.waves} in \textsection\ref{s.wave}. 

\subsection{Comparison of avalanches}

The indicator on the right side of \eqref{e.main} is a kind of size-biasing. To make this explicit, let us compare the burst size of a stationary avalanche (that is, the number of chips lost to the sink when stabilizing $\eta+\delta_i$ with respect to $\Stab_z$, where $\eta$ is uniformly distributed on $\Rec(z)$) to that of the threshold avalanche (the number of chips lost to the sink when stabilizing $\rho_{\tau-1}+\delta_{i_\tau}$ to $\rho_\tau$). For $b=0,1,2,\ldots$ let $p_b$ denote the stationary probability of an avalanche of burst size~$b$,
	\[ p_b := \sum_{i \in V} \sum_{\eta \in \Rec(z)} \frac{\alpha_i}{\kappa} \one \{ \av_{i \to z} (\eta) = b \}. \]
Since the expected burst size is $1$ in stationarity, $\sum_{b \geq 0} bp_b =1$. 

Let $q_b = q_b(s_0)$ denote the probability that the \emph{threshold} avalanche
has burst size~$b$,
	\[ q_b := \PP_{s_0} \{ \av_{i_\tau \to z} (\rho_{\tau}) = b \}. \]
Then according to Theorem~\ref{t.main}, in the limit $|s_0| \to -\infty$ we have
	\[ q_b \to \sum_{i \in V} \sum_{\rho \in \Rec(z)}  \frac{\alpha_i}{\kappa} \sum_{m \geq 0} \one \{m \leq \av_{i \to z}(\rho)-1 = b-1 \}. \]
The sum over $m$ equals $b \one \{\av_{i \to z}(\rho)=b\}$, so we have shown the following.

\begin{corollary}
\label{c.avalanche}
$q_b \to bp_b$ as $|s_0| \to -\infty$.
\end{corollary}
	
In particular, the expected burst size of the threshold avalanche converges to the second moment of the burst size of a stationary avalanche:
	\[ \EE_{s_0} [\av_{i_\tau \to z} (\rho_{\tau})]
	= \sum_b b q_b 
	\to \sum_b b^2 p_b \]
as $|s_0| \to -\infty$. If the burst size of a stationary avalanche follows a truncated power law $p_b \sim b^{-\beta}$, then burst size of the threshold avalanche follows the heavier-tailed power law $q_b \sim b^{1-\beta}$.

\section{\texorpdfstring{The $z$-recurrent decomposition}{The z-recurrent decomposition}}

In this section we prove existence and uniqueness of the decomposition \eqref{e.zrec.intro} and track how it changes during a single time step of the closed chain. Given $z \in V$ we say that sandpiles $s_1, s_2 : V \to \Z$ are \emph{$z$-equivalent} if $s_1 -s_2 = \Delta v + m \delta_z$ for some $v: V \to \Z$ and $m \in \Z$.  Here $\Delta$ is the graph Laplacian \eqref{e.laplacian} and $\delta_z$ denotes the configuration with a single chip at vertex $z$.

We will need a few well-known facts about sandpiles.

\begin{lemma}
\label{l.basics}
Let $z\in V$ and $s: V \to \Z$.
\begin{enumerate} 
\item[(a)] $s$ is $z$-equivalent to a unique $\rho \in \Rec(z)$.
\item[(b)] 
$s$ is stabilizable if and only if there exists $u: V\to \N$ such that 
	\begin{equation} \label{e.leastaction} s + \Delta u \leq \deg -1. \end{equation}
\item[(c)] If $s$ is stabilizable then $s - \Delta v$ is stabilizable for all $v : V \to \Z$.
\item[(d)] If $s \leq s'$ and $s'$ is stabilizable, then $s$ is stabilizable.
\end{enumerate}
\end{lemma}

Item (a) was remarked by Dhar \cite{Dhar90}; for a proof see \cite[Lemmas 2.13 and 2.15]{HLMPPW}.
Regarding item (b), if $s$ is stabilizable then its odometer is the pointwise smallest function $u$ satisfying \eqref{e.leastaction}. This ``least action principle'' was used in \cite{FLP} to bound the growth rates of sandpiles on $\Z^d$. Pegden and Smart \cite{PS} used it to prove existence of the scaling limit of the abelian sandpile on $\Z^d$.  In \cite{BL} the least action principle is proved for a more general class of processes called abelian networks. In this paper we will not need (b) itself but only its immediate consequences (c) and (d).

\begin{lemma}
\label{l.rectest}
\moniker{$z$-Recurrent Decomposition}
Given $s: V \to \Z$ and $z \in V$, there is a unique triple $(\rho,m,v)$ where $\rho \in \Rec(z)$ and $m \in \Z$ and $v : V \to \Z$ with $v(z) = 0$, such that
	\begin{equation} \label{e.zrec} s = \rho + m \delta_z + \Delta v. \end{equation}
Moreover, $s$ is stabilizable 
if and only if $m < 0$.
\end{lemma}

\begin{proof}
By Lemma~\ref{l.basics}(a) there is a unique $\rho \in \Rec(z)$ that is $z$-equivalent to $s$. Then
	\[ s = \rho + m\delta_z + \Delta v \]
for some $m \in \Z$ and $v: V \to \Z$. By adding a constant to $v$ we can ensure that $v(z) = 0$. 
Next we verify the uniqueness of this decomposition.
Counting chips shows that $m$ is uniquely determined: $m = |s| - |\rho|$ since $|\Delta v | =0$.  Since $G$ is connected Eulerian, the kernel of $\Delta$ is one-dimensional consisting of the constant functions, so $v$ is uniquely determined subject to the condition $v(z) = 0$.

It remains to show that $s$ is stabilizable if and only if $m<0$. By Lemma~\ref{l.basics}(c), $s$ is stabilizable if and only if $\rho+m\delta_z$ is stabilizable. By Lemma~\ref{l.basics}(d) it therefore suffices to show that $\rho-\delta_z$ is stabilizable and that $\rho$ is not stabilizable.  Recalling from Definition~\ref{d.zrec} that $\rho(z) = \deg(z)$ and $\rho(i) < \deg(i)$ for all $i\neq z$ we see that $\rho - \delta_z$ is trivially stabilizable (it has no unstable vertices). On the other hand $\rho$ has one unstable vertex, $z$. By Dhar's burning test, toppling $z$ results in an avalanche in which every other site topples exactly once, yielding $\rho$ again (here we have used that $\Delta \one = 0$ since $G$ is Eulerian). Therefore $\rho$ is not stabilizable.
\end{proof}

The next lemma tracks how the $z$-recurrent decomposition changes when we apply an addition operator $a_i$. Given $s: V \to \Z$ and $i \in V$, let $u$ be the odometer for $s+\delta_i$ with respect to $\Stab$ (or $u\equiv 0$ if $s+\delta_i$ is not stabilizable). Denoting by $R_z(s)$ the unique $\rho \in \Rec(z)$ satisfying \eqref{e.zrec}, let $\hat{u}$ be the odometer for $R_z(s)+\delta_i$ with respect to $\Stab_z$. 

\begin{lemma}
\label{l.intertwining}
Fix $i, z \in V$. Let $s$ be a sandpile with $z$-recurrent decomposition \eqref{e.zrec}.
Then $a_i s$ has the $z$-recurrent decomposition 
\begin{align*} a_i s  = \hat{a}_i \rho + (m+\beta)\delta_z + \Delta (v + u-\hat{u}-u(z)) \end{align*}
where $\beta = \av_{i \to z}(\hat{a}_i \rho)$ is the burst size {\em (Definition~\ref{d.burst})}.
\end{lemma}

In particular, it follows that $R_z$ intertwines the open and closed addition operators as claimed in \eqref{e.intertwining}: $R_z(a_i s) = \hat{a}_i \rho = \hat{a}_i R_z(s)$. 

\begin{proof}[Proof of Lemma~\ref{l.intertwining}]
By the definition of $a_i$ we have
	\begin{align}  a_i s &= s + \delta_i + \Delta u. \label{e.closedodom} \end{align}
By the definition of $\hat{a}_i$, since $\beta$ chips fall into the sink $z$ when we stabilize $\rho + \delta_i$ with respect to $\Stab_z$, we have
	\begin{align}  \hat{a}_i \rho &= \rho  + \delta_i + \Delta \hat{u} - \beta \delta_z \label{e.openodom} 
			\end{align}
where the last term removes the extra chips at the sink in order to restore the condition $\hat{a}_i \rho(z) = \deg(z)$ (recall that $\rho$ and $\hat{a}_i \rho$ belong to $\Rec(z)$, Definition~\ref{d.zrec}).

Substituting equation \eqref{e.zrec} into \eqref{e.closedodom} and then using \eqref{e.openodom}, we find
\begin{align*} a_i s         &= \rho + \delta_i + m\delta_z + \Delta (v+u) \\
				      &= \hat{a}_i \rho + (m+\beta)\delta_z + \Delta (v+u-\hat{u}-c). \end{align*}
The last equality holds for any constant $c$ since $\Delta c = 0$. Taking $c=u(z)$ so that the function $u-\hat{u}+v-c$ vanishes at $z$, the last line is the $z$-recurrent decomposition of $a_i s$. 
\end{proof}

\section{A Markov renewal theorem}

We will ultimately apply the Markov renewal theorem of this section to a variant of the open chain $\rho_k$ of \textsection\ref{s.open}, but we state it here in more generality.

Let $(X_k)_{k \geq 0}$ be an irreducible Markov chain with a finite state space.
Write $P(\cdot,\cdot)$ for the transition matrix and $\pi(\cdot)$ for the stationary distribution of $X_k$. Let $\mathcal{E} = \{(x,y) \,:\, P(x,y)>0\}$.
Suppose that each edge $(x,y) \in \mathcal{E}$ has an associated ``length'' $\ell(x,y)$, which is a nonnegative integer.  We require that the pair $(P,\ell)$ is aperiodic in the sense that
	\begin{equation} \label{e.aperiodic} \gcd \left \{ \sum_{j=1}^{k} \ell(x_{j-1},x_{j}) \,:\, k \geq 1, \, x_0=x_k,\, (x_{j-1},x_{j}) \in \mathcal{E} \; \forall j \right \} = 1. \end{equation}
This implies that for all sufficiently large $L \in \N$ there is a closed path with edges in $\mathcal{E}$ of total length $L$.
%

Write $\PP_{x_0}$ for the law of $(X_k)_{k \geq 0}$ conditioned on $X_0 = x_0$. Let
	\[ \lambda_k = \sum_{i=1}^k \ell(X_{i-1},X_i) \]
be the total length of the edges traversed up to time $k$. For fixed $n \in \N$ consider the random time
	\[ \tau = \tau_n = \min \{k \,:\, \lambda_k \geq n \}. \]

\begin{prop}
\label{t.marathon}
\moniker{Markov Renewal Theorem}
Suppose that $(X_k)_{k \geq 0}$ is an irreducible Markov chain with length function $\ell$ satisfying \eqref{e.aperiodic}.
For any states $x_0,x,y$ and any $m \in \N$,
	\[ \PP_{x_0} \large\{ (X_{\tau_n-1},X_{\tau_n},\lambda_{\tau_n}-n)= (x,y,m) \large\} \to \frac{1}{Z} \pi(x)P(x,y)\one \{0 \leq m \leq \ell(x,y)-1\} \]
as $n \to \infty$, where the normalizing constant $Z$ equals $\sum_{(x,y) \in \mathcal{E}} \pi(x) P(x,y) \ell(x,y)$.
\end{prop}

A helpful metaphor for the triple $(X_{\tau-1},X_\tau,\lambda_\tau-n)$ is the finish line of a marathon of length $n$. (Unlike a real marathon, this one has a random route following a Markov chain!)
According to Proposition~\ref{t.marathon}, in the limit as $n \to \infty$ the distribution of the finish line does not depend on the starting point $x_0$. The distribution of the finishing edge (of a very long marathon) is the size-biasing by $\ell$ of the edge stationary distribution $\pi(x) P(x,y)$; and conditioned on the finish line being on a given edge $(x,y)$ it is uniformly distributed along that edge. 

In our application to sandpiles, it will be important to allow some edge lengths to be zero. Of course, if $\ell(x,y)=0$ then $\PP_{x_0} \{ (X_{\tau-1}, X_\tau) = (x,y) \}=0$.

For a much more general Markov renewal theorem, see Kesten \cite{Kesten}. Below we include a proof of Proposition~\ref{t.marathon} for the sake of completeness.  To motivate the proof, consider first a special case: Setting $\ell(x,y) = \ell(y,x) = 1$ and all other edge lengths to zero, Proposition~\ref{t.marathon} reduces to the following.

\begin{corollary}
\label{c.crossings}
Let $X_k$ be an irreducible and aperiodic Markov chain with transition matrix $P(\cdot,\cdot)$ and stationary distribution $\pi(\cdot)$.
Let $\tau_0 = 0$ and $\tau_n = \min \{ k> \tau_{n-1} \,:\, \{X_{k-1},X_k\} = \{x,y\} \}$ be the $n$-th crossing of edge $\{x,y\}$, counting crossings in both directions. Then as $n \to \infty$,
	\[ \PP_{x_0}(X_{\tau_n} = y) \to \frac{\pi(x)P(x,y)}{\pi(x)P(x,y) + \pi(y)P(y,x)}. \]
\end{corollary}

Next we recall two basic facts about discrete time Markov chains with a finite state space, Lemmas~\ref{l.convergence} and~\ref{l.watched}. These will immediately imply Corollary~\ref{c.crossings}, and with a little more effort the full Proposition~\ref{t.marathon}.

\begin{lemma}
\label{l.convergence}
\moniker{Convergence Theorem}
If $(Y_n)_{n \geq 0}$ is an irreducible and aperiodic Markov chain with stationary distribution $\tilde \pi$, then for any $y_0$ and $y$, 
	\[ \PP_{y_0} (Y_n = y) \to \tilde \pi(y) \]
as $n \to \infty$.
\end{lemma}

If $Y$ is a Markov chain and $A$ is a subset of its state space, then the \emph{chain watched only on $A$} is given by
	\[ (Y|_A)_n =  Y_{a_n} \]
where $0 \leq a_1 < a_2 < \cdots$ are the times for which $Y_{a_n} \in A$.  The following identity is a consequence of the ergodic theorem for discrete time Markov chains \cite[ch.\ 2 eq.\ (27)]{AFbook}. 
\begin{lemma} 
\label{l.watched}
\moniker{Chain Watched Only On $A$}
If $Y$ is an irreducible Markov chain with stationary distribution $\tilde \pi$, then the stationary distribution $\pi_A$ of $Y|_A$ is
	\[ \pi_A(x) = \frac{\tilde \pi(x)}{\tilde \pi(A)}. \]
\end{lemma}

Taking $Y_n = (X_n,X_{n+1})$ and $A = \{(x,y),(y,x)\}$, the convergence theorem applied to the chain $Y|_A$ proves Corollary~\ref{c.crossings}.

To adapt this argument to prove the full Proposition~\ref{t.marathon}, observe that the convergence theorem applies to a \emph{deterministic} time $n$ whereas the renewal theorem involves the random time $\tau_n$. Our strategy will be to define a chain $\xi$ whose state at time $n$ tracks the behavior of our original chain $X$ at time $\tau_n$. We then use Lemma~\ref{l.watched} to compute the stationary distribution of $\xi$ and appeal to the convergence theorem for $\xi$.

\begin{proof}[Proof of Proposition~\ref{t.marathon}]
Define a discrete time Markov chain $Y$ on state space $\mathcal{Y} = \{(x,y,j) \,:\, (x,y) \in \mathcal{E}, \, j \in \N, \, 0 \leq j \leq \ell(x,y)\}$ with initial state $Y_0 = (X_0,X_1,0)$ and transition probabilities given by
	\begin{align*} &\tilde P((x,y,j),(x,y,j+1)) = 1,  && j=0,1,\ldots,\ell(x,y)-1 \\
	  &\tilde P((x,y,j),(y,z,0)) = P(y,z),  && j = \ell(x,y). \end{align*}
Irreducibility of $X$ implies irreducibility of $Y$.

Let us compute the stationary distribution $\tilde \pi$ of $Y$. Noting that $\tilde \pi (x,y,j) = \tilde \pi (x,y,j-1)$ for all $j=1,\ldots,\ell(x,y)$, it suffices to compute $\tilde{\pi}(x,y,0)$. The chain $Y$ watched only on $\mathcal{E} \times \{0\}$ has the same law as the edge chain $((X_k,X_{k+1},0))_{k \geq 0}$, which has stationary distribution $\pi(x)P(x,y)$.
Writing $Z_0 = 1/\tilde{\pi}(\mathcal{E} \times \{0\})$ we obtain $\pi(x)P(x,y) = Z_0 \tilde \pi (x,y,0)$ by Lemma~\ref{l.watched}, hence
	\[ \tilde \pi (x,y,j) = \tilde \pi (x,y,0) = \frac{1}{Z_0} \pi(x)P(x,y). \]

Now consider the chain $\xi = Y |_{\mathcal{Y} - \mathcal{E} \times \{0\}}$. One time step in $\xi$ corresponds to one unit of \emph{length} traveled by the original chain $X$, so 
	$\xi_n = (X_{\tau-1},X_\tau,n-\lambda_{\tau-1})$. 
Note that for any state $x_0$ of $X$,
	\begin{equation} \label{e.minusuniform} \PP_{x_0} \{(X_\tau-1, X_\tau, \lambda_\tau-n) = (x,y,m) \} = \PP_{x_0} \{ \xi_n = (x,y,\ell(x,y)-m) \} \end{equation}
since $\lambda_\tau - \lambda_{\tau-1} = \ell(x,y)$ on the event $\{(X_{\tau-1},X_\tau) = (x,y)\}$.

 By Lemma~\ref{l.watched}, the stationary distribution of $\xi$ is
	\[ \pi_\xi (x,y,j) = \frac{\tilde \pi(x,y,j)}{\tilde \pi(\mathcal{Y} - \mathcal{E} \times \{0\})} 
	=  \frac{1}{Z} \pi(x)P(x,y), \qquad j=1,\ldots,\ell(x,y). \]
Moreover, $\xi$ is irreducible (since $X$ is) and aperiodic by \eqref{e.aperiodic}. The result now follows from \eqref{e.minusuniform} by applying the convergence theorem, Lemma~\ref{l.convergence}, to $\xi$.
\end{proof}

\section{\texorpdfstring{Proof of Theorem~\ref{t.main} and Conjecture~\ref{c.pppr}}{Proof of Theorem 2 and Conjecture 1}}
\label{s.proof}

Consider the Markov chain $X_k = (i_k,\rho_k)$ where $\rho_k$ is the open chain (\textsection\ref{s.open}). Recalling our assumption that $\alpha_i>0$ for all $i$, the chain $X$ is irreducible (see \cite[Lemma~2.17]{HLMPPW}). Its transition matrix and stationary distribution are given by
	\[ P((i',\rho'),(i,\rho)) = \alpha_i \one \{\rho = \hat{a}_i \rho'\}, \qquad \pi((i,\rho)) = \frac{\alpha_i}{\kappa}. \]
For $\rho = \hat{a}_i \rho'$ we take length function 
	\begin{equation} \label{e.length} \ell((i',\rho'), (i,\rho)) = \av_{i\to z}(\rho) = |\rho'| - |\rho| + 1. \end{equation}
The aperiodicity condition \eqref{e.aperiodic} is trivially satisfied because the chain has loops of length one: $\hat{a}_z \rho = \rho$ and $\ell((z,\rho),(z,\rho)) = 1$.  

The normalizing constant of Proposition~\ref{t.marathon},
		\begin{align*} Z 
		&= \sum_{i,i' \in V} \sum_{\rho \in \Rec(z)} \frac{\alpha_{i'}}{\kappa} \alpha_i \av_{i \to z}(\rho)
		\end{align*}
has the interpretation of the mean burst size in stationarity, which must be $1$ by conservation of chips. To verify this formally, since $\hat{a}_i$ is a permutation of $\Rec(z)$, we have
	\[ \sum_{\rho \in \Rec(z)} \av_{i\to z} (\hat{a}_i \rho) = \sum_{\rho} |\rho| - \sum_{\rho} |\hat{a}_i \rho| + \sum_\rho 1 = \# \Rec(z) = \kappa, \]
so $Z=(\sum_{i \in V} \alpha_i )^2 = 1$. 

We now turn to the proof of our main theorem.

\begin{proof}[Proof of Theorem~\ref{t.main}]
By Lemma~\ref{l.rectest} the threshold time equals
	\[ \tau = \min \{k \,:\, m_k \geq 0 \}. \]
By Lemma~\ref{l.intertwining},
	\[ m_k = m_0 + \sum_{j=1}^k \av_{i_j \to z}(\rho_{j}). \]
Now we use Proposition~\ref{t.marathon} for the Markov chain $X_k = (i_k,\rho_k)$ with $n = -m_0$ and length function \eqref{e.length}. Since $|s_0| = |\rho_0| + m_0$ and $0 \leq |\rho_0| \leq \# E$ we have $n \to \infty$ as $|s_0| \to -\infty$. Hence
	\begin{multline*}
	\PP_{s_0} \{ (i_{\tau-1},\rho_{\tau-1},i_\tau, \rho_\tau, m_\tau) = (i',\rho',i,\rho,m) \} \\
	 	\begin{aligned}
	 &= \PP_{s_0} \{ (X_{\tau-1},X_\tau,\lambda_\tau-n) = ((i',\rho'), (i,\rho), m) \} \\
	 &\to \frac{1}{Z} \frac{\alpha_{i'}}{\kappa} \alpha_i \one \{\rho = \hat{a}_i \rho' \} \one \{0 \leq m \leq |\rho'| - |\rho| \}
	 	\end{aligned}
	\end{multline*}
as $|s_0| \to -\infty$. By the remark preceding the proof, $Z=1$. Summing over $i'$ and $\rho'$ yields Theorem~\ref{t.main}.
\end{proof}

We remark that the above proof also shows that the random vertex $i_{\tau-1}$ is independent of the triple $(i_\tau, \rho_\tau, m_\tau)$ in the $|s_0| \to -\infty$ limit.

Corollaries~\ref{c.sinkatsource},~\ref{c.straw},~\ref{c.marginal} and~\ref{c.avalanche} follow readily from Theorem~\ref{t.main}. We conclude this section by proving Corollary~\ref{c.densities} and hence Conjecture~\ref{c.pppr}.

\begin{proof}[Proof of Corollary~\ref{c.densities}]
Consider the $z$-recurrent decomposition of the closed chain
	\[ s_k = \rho_k^z + m_k^z \delta_z + \Delta v_k^z \]
Since $| \Delta v_k^z | = 0$ we have $|s_k| = |\rho_k^z| + m_k^z$.  Conditioned on $i_\tau = z$ we have $m_\tau^z = 0$ and hence $|s_\tau| = |\rho_\tau^z|$. Now
	\begin{align} \PP_{s_0} \{ |s_\tau|=n \} &= \sum_{z \in V} \PP_{s_0} \{ i_\tau=z,\, |s_\tau| = n \} \nonumber \\
		&= \sum_{z \in V} \PP_{s_0} \{i_\tau = z,\, |\rho_\tau^z|=n \} \label{e.onestepaway} 
	\end{align}
By Corollary~\ref{c.sinkatsource}, since $\rho_\tau^z = R_z(s_\tau)$, the right side of \eqref{e.onestepaway} converges as $|s_0| \to -\infty$ to
	\[ \sum_{z \in V} \frac{\alpha_z}{\kappa} \sum_{\rho \in \Rec(z)} \one \{ |\rho| = n \}. \]
By the theorem of Perrot and Pham \cite{PP}, the inner sum does not depend on $z$.
\end{proof}

\section{The threshold wave}
\label{s.wave}

In this section we assume 
that $G$ is undirected.
Dhar and Manna \cite{DM} and Ivashkevich, Ktitarev and Preizzhev \cite{IKP} introduced a decomposition of abelian sandpile avalanches into smaller toppling events called \emph{waves}: given $\rho \in \Rec(z)$, during each wave in stabilizing $\rho + \delta_i$, the source $i$ topples once and then the resulting configuration is stabilized on $V-\{i,z\}$. This procedure is repeated until vertex $i$ becomes stable. To restate this more formally, we can define a wave operator
	\[ \Wave(\eta) = \Stab_{\{i,z\}}(\eta + \Delta \delta_i) \]
acting on the set of $\eta: V \to \Z$ such that exactly one vertex $i \in V-\{z\}$ satisfies $\eta(i) \geq \deg(i)$; the subscript $\{i,z\}$ indicates that $i$ and $z$ are forbidden to topple. 
A \emph{zeroth wave} from $i$ to $z$ is a pair $(\rho,\rho+\delta_i)$ where $\rho \in \Rec(z)$. A $j$th wave from $i$ to $z$, for $j\geq 1$, is a pair $(\Wave^{j-1}(\rho+\delta_i),\Wave^{j}(\rho+\delta_i))$ such that $\rho \in \Rec(z)$ and $\Wave^{j-1}(\rho+\delta_i)(i) = \deg(i)$. This pair is called a \emph{last wave} if $\Wave^j(\rho+\delta_i)(i)<\deg(i)$.

Ivashkevich, Ktitarev and Priezzhev \cite{IKP} extended the burning bijection of Majumdar and Dhar \cite{MD} to give a bijection between the set of waves
from $i$ to $z$ 
and the set of spanning forests of $G$ with (at most) two components, a possibly empty component $t_i$ rooted at $i$ and a nonempty component $t_z$ rooted at $z$ 
(Table~\ref{table.waves}). 
The vertex set of $t_i$ is the set of sites that topple during the corresponding wave. Each site topples at most once per wave, so the number of chips falling into the sink equals the number of edges of $G$ adjoining $t_i$ with $z$.  For a wave resulting in configuration $\eta$ we denote this number by $b_{i\to z}(\eta)$; it is the analogue of burst size for waves.

\begin{table}[here]
\begin{tabular}{rcl}
waves from $i$ to $z$ & \; $\longleftrightarrow$ \; & forests $t_i \cup t_z$ \\
zeroth waves from $i$ to $z$ & \; $\longleftrightarrow$ \; & trees $t_i \cup t_z$ with $t_i = \emptyset$ \\
last waves from $i$ to $z$ & \; $\longleftrightarrow$ \; &  forests $t_i \cup t_z$ with $i \sim t_z$ \\
bursting waves from $i$ to $z$ & \; $\longleftrightarrow$ \; & forests $t_i \cup t_z$ with $z \sim t_i$
\end{tabular}
\medskip
\caption{\small Bijections between waves and $2$-component spanning forests \cite{IKP}. A \emph{bursting wave} is one in which chips fall into the sink.}
\label{table.waves}
\end{table}

Consider a refinement $(\eta_j)_{j \geq 0}$ of the open chain in which a single time step consists of performing one wave (i.e., applying $\Wave$) if vertex $i_j$ is unstable, and otherwise adding one chip at an independent random vertex $i_{j+1}$. (In the former case we keep $i_{j+1}=i_j$.)
The process $m_k$ of \eqref{e.zrec.closed} has a refinement 
keeping track of the total number of chips falling into the sink,
	\[ m_j = m_0 + \sum_{r=1}^j b_{i_{r} \to z}(\eta_r). \]
In view of \eqref{e.thresholdm} it is natural to define the threshold of the refined chain as 
$\tau = \min \{ j \,:\, m_j \geq 0 \}$. This $\tau$ singles out a particular bursting wave of the threshold avalanche.
Applying the Markov renewal theorem (Proposition~\ref{t.marathon}) to the chain $(i_j,\eta_j)$, we find
	\begin{equation} \label{e.waveburst} \PP_{s_0} \{ (i_\tau,\eta_\tau,m_\tau) = (i,\eta,m) \} \to \frac{\alpha_i}{Z} \one \{0 \leq m \leq b_{i \to z}(\eta)-1 \} \end{equation}
as $|s_0| \to -\infty$.

Let us see how to interpret this result in terms of the uniform spanning tree and in the process compute the normalizing constant $Z$. 
The triples $(i,\eta,m)$ for which the right side of \eqref{e.waveburst} does not vanish are of two types.  First, there are the triples with $(z,\eta,0)$ for $\eta \in \Rec(z)$, which correspond to dropping a chip directly into the sink. 
After choosing an arbitrary ordering of the edges incident to $z$, the remaining triples are in bijection with pairs $(t_i \cup t_z, e)$ where $t_i \cup t_z$ is a $2$-component spanning forest rooted at $\{i,z\}$ and $e$ is an edge adjoining $t_i$ with $z$.  Then $t_i \cup t_z \cup \{e\}$ is a spanning tree.  Conversely, given $i \neq z$, every spanning tree $t$ decomposes uniquely as $t_i \cup t_z \cup \{e\}$ where $e$ is the last edge on the path from $i$ to $z$ in $t$. Therefore $Z=\kappa$, the number of spanning trees of $G$. Denoting by $A_\tau$ the set of sites that topple during the threshold wave, we find that
	\begin{equation} \label{e.topplingset} \PP_{s_0} \{ i_\tau = i, A_\tau = A \} \to \frac{\alpha_i}{\kappa} \sum_{t} \one \{ t_i = A \}  \end{equation}
as $|s_0| \to -\infty$, where the sum is over spanning trees $t$ of $G$, and $t_i$ is the set of vertices $i'$ such that the paths $i\to z$ and $i' \to z$ in $t$ have the same last edge.


\section{Infinite volume limits}

Athreya and J{\'a}rai \cite{AJ} showed that the stationary distribution of the open chain on finite subsets of $\Z^d$ has a limit which is a measure on sandpiles on all of $\Z^d$. It would be interesting to prove that the $|s_0| \to -\infty$ limit of the threshold state (or its recurrent representative, whose distribution is given by Corollary~\ref{c.marginal}) has such an infinite volume limit on $\Z^d$.  There are actually three different limits one could take: epicenter at origin, sink at origin, or neither at origin.  The case when both epicenter and sink are at the origin is exactly the stationary state by Corollary~\ref{c.sinkatsource}.

In addition, it may be time to revisit the questions about stabilizability in infinite volume posed by Fey, Meester and Redig \cite{FMR}.

\section*{Acknowledgments}

The author wishes to thank Antal J{\'a}rai, Yuval Peres, Laurent Saloff-Coste and David Wilson for inspiring conversations.

\end{document}